\documentclass[11pt]{amsart}
\usepackage[margin=1.25in]{geometry}
 \usepackage{amsmath,amssymb,amsthm,scalerel,hyperref}
\usepackage{tikz}
\usepackage{tikz-cd}
\usetikzlibrary{calc}
\usepackage{ytableau}

\newtheorem{thm}{Theorem}[section]
\newtheorem{prop}[thm]{Proposition}
\newtheorem{lem}[thm]{Lemma}

\theoremstyle{definition}
\newtheorem{definition}[thm]{Definition}

\DeclareMathOperator{\im}{Im}
\DeclareMathOperator{\sgn}{sgn}
\DeclareMathOperator{\ind}{Ind}

\newcommand{\Tlam}{\mathcal{T}_\lambda}
\newcommand{\tilMlam}{\tilde{M}^{\lambda}}

\newcount\tableauRow\newcount\tableauCol
\newcommand\Tableau[1]{%
  \begin{tikzpicture}[scale=0.5,draw/.append style={thick,black},baseline=4mm]
    \tableauRow=0
    \foreach \Row in {#1} {
       \tableauCol=1
       \foreach\k in \Row {
          \draw(\the\tableauCol,\the\tableauRow)+(-.5,-.5)rectangle++(.5,.5);
          \draw(\the\tableauCol,\the\tableauRow)node{\k};
          \global\advance\tableauCol by 1
       }
       \global\advance\tableauRow by -1
    }
  \end{tikzpicture}
}
\newcommand\Tabloid[1]{%
  \begin{tikzpicture}[scale=0.5,draw/.append style={thick,black},baseline=4mm]
    \tableauRow=0
    \foreach \Row in {#1} {
       \tableauCol=1
       \foreach\k in \Row {
          \draw($(\the\tableauCol,\the\tableauRow)+(-.5,-.5)$)--++(1,0);
          \draw($(\the\tableauCol,\the\tableauRow)+(-.5,.5)$)--++(1,0);
          \draw(\the\tableauCol,\the\tableauRow)node{\k};
          \global\advance\tableauCol by 1
       }
       \global\advance\tableauRow by -1
    }
  \end{tikzpicture}
}
\newcommand\ColumnTabloid[1]{%
  \begin{tikzpicture}[scale=.5,draw/.append style={thick,black},baseline=4mm]
    \tableauRow=0
    \foreach \Row in {#1} {
       \tableauCol=1
       \foreach\k in \Row {
          \draw($(\the\tableauCol,\the\tableauRow)+(-.5,-.5)$)--++(0,1);
          \draw($(\the\tableauCol,\the\tableauRow)+(.5,-.5)$)--++(0,1);
          \draw(\the\tableauCol,\the\tableauRow)node{\k};
          \global\advance\tableauCol by 1
       }
       \global\advance\tableauRow by -1
    }
  \end{tikzpicture}
}

\usepackage{tikz}
\usetikzlibrary{matrix}
\usepackage[utf8]{inputenc}
\newcommand{\Addresses}{{
  \bigskip
  \footnotesize

  S.~Brauner, \textsc{Department of Mathematics, University of Minnesota,
    Twin Cities}\par\nopagebreak
  \textit{E-mail address} : \texttt{braun622@umn.edu}

  \medskip

  T.~Friedmann, \textsc{Department of Mathematics and Statistics, Colby College}\par\nopagebreak
  \textit{E-mail address} : \texttt{tamar.friedmann@colby.edu}

}}


\title{A Simplified Presentation of Specht Modules}
\author{Sarah Brauner and Tamar Friedmann}
  
  \date{}

\begin{document}
\maketitle
\begin{abstract}

Fulton and Kraskiewicz gave a presentation of Specht modules as a quotient of the space of column tabloids by dual Garnir relations. We simplify this presentation by showing that it
can be generated by a single relation for each pair of columns of a tableau with ordered columns, thereby significantly reducing the number of generators given in the original construction.
Our presentation applies to all Specht modules, and is of a similar nature to a recent result by Friedmann-Hanlon-Stanley-Wachs that applies to staircase partitions. We show that our presentation implies the Friedmann-Hanlon-Stanley-Wachs presentation.
\end{abstract}

\section{Introduction}
Representations of the symmetric group $S_m$ have a long and beautiful history in mathematics. Partitions of $m$ biject with the irreducible representations of $S_{m}$ given by Specht modules; these representations have a basis corresponding to standard Young tableaux. The relations that allow us to express any tableau as a linear combination of standard Young tableaux are called Garnir relations. 

For the remainder of this paper, we will work over $\mathbb{C}$. For a partition $\lambda =  (\lambda_{1}, \dots , \lambda_k)$ of $m$, let $\lambda^{'} = (\lambda^{'}_{1}, \dots , \lambda^{'}_{j})$ be the conjugate of $\lambda$
and let $S^{\lambda}$ be the Specht module corresponding to $\lambda.$  Also, let $\mathcal{T}_{\lambda}$ be the set of Young tableaux of shape $\lambda$ in which each element of $[m]$ appears exactly once.
For any $t \in \mathcal{T}_{\lambda}$, let $R_{t}$ be the row stabilizer of $t$, let $C_{t}$ be the column stabilizer of $t$, let $\{ t \} $ be the associated row tabloid, and let
 \[ \varepsilon_{t} = \sum_{\beta \in C_{t}} \sgn(\beta) \{\beta t \} \] be the associated row polytabloid of $t$. 
It is a classical result that the set of all $\varepsilon_{t}$ where $t$ is a standard Young tableau forms a basis of $S^{\lambda}$.

In \cite{Kras} and \cite{fulton}, both Kraskiewicz and Fulton introduce a dual construction of the Specht module, $\tilde{S^{\lambda}},$ using column tabloids rather than row tabloids. Column tabloids are quite similar to row tabloids: a row tabloid is an equivalence class of numberings of a Young diagram such that two row tabloids are equivalent if they have the same entries in each row. Dually, a \emph{column tabloid}, denoted $[t]$, is an equivalence class of numberings of a Young diagram such that two column tabloids are equivalent \emph{up to sign} if they have the same entries in each column. Herein lies a key difference between row and column tabloids: unlike row tabloids, column tabloids are antisymmetric within columns. That is, for a column tabloid $[t]$ and $\beta \in C_{t}$, we have $[t] = \sgn(\beta) \beta [t] = \sgn(\beta) [\beta t].$

Let $\tilde{M}^{\lambda}$ be the vector space generated by all $[t]$ where $t$ is a Young tableau of shape $\lambda$, modulo the antisymmetry relations which are generated by $[t] - \sgn(\beta)[\beta t]$ for each $t$ and $\beta \in C_{t}$. Thus a basis of $\tilde{M}^{\lambda}$ is given by all column-ordered tabloids of shape $\lambda$, where by ``ordered" we mean that the numbers in the tableaux increase going down the columns. 

The symmetric group acts on $[t] \in \tilde{M}^{\lambda}$ in the natural way: $\sigma [t] = [\sigma t].$ Fulton defines $\tilde{S^{\lambda}}$ to be the subspace of $\tilde{M}^{\lambda}$ spanned by elements of the form $\sum_{\alpha \in R_{t}} \alpha [t]$. He shows that this dual construction of a Specht module is isomorphic to its row tabloid counterpart, $S^{\lambda}$ \cite{fulton}.

In order to prove this result, Fulton defines a dual straightening algorithm which gives a presentation of Specht modules as a quotient space of $\tilde{M}^{\lambda}$ by dual Garnir relations. This presentation also appeared in \cite{Kras} two years earlier, and is discussed in Section \ref{spechtconstruction}. There is a dual Garnir relation for each $t\in \Tlam$, each choice of adjacent columns, and each $k$ up to the length of the next column. In Section \ref{newSp}, we simplify this presentation significantly by showing that over a field with characteristic 0, we need only a single relation called $\eta$ for each choice of adjacent columns of an ordered column tabloid $[t]\in \tilMlam$ (Theorem \ref{spechtgarnir}). Our result applies to all partitions. A simplification that applied to staircase partitions was previously achieved in \cite[Corollary 3.2]{FHSW} by Friedmann-Hanlon-Stanley-Wachs\footnote{In \cite{FHSW3}, the authors give a generalization of their result from staircase partitions to partitions of strictly decreasing column lengths.}; their presentation and methods were the motivation for our work. We show in Proposition \ref{kernelalpha} that our simplified presentation implies the Friedmann-Hanlon-Stanley-Wachs presentation. 


\section*{Acknowledgements}  The authors would like to thank Darij Grinberg, Marissa Miller, Leslie Nordstrom, Vic Reiner, and Michelle Wachs for useful conversations and an anonymous referee for helpful comments. 
Some of this work was carried out while the authors were at Smith College; the authors gratefully acknowledge the National Science Foundation (DMS-1143716) and Smith College for their support of the Center for Women in Mathematics and the Post-Baccalaureate Program. This work was partially funded by a Smith College Summer Research Fund, and SB was supported by the NSF Graduate Research Fellowship (Award Number DMS-0007404) during the drafting of the manuscript. 
\vskip 1cm

\section{A presentation of Specht modules using column tabloids}\label{spechtconstruction}

\subsection{Garnir relations for column tabloids}
In this section we recall known presentations of Specht modules in terms of dual Garnir relations. 

In \cite[Ch. 7.2]{fulton}, Fulton introduces a map  $$\alpha: \tilde{M}^{\lambda} \to S^{\lambda}$$  given by 
\[ \alpha: [t] \mapsto \varepsilon_{t}. \]
The map $\alpha$ is $S_{m}$-equivariant and surjective. Moreover, $\ker(\alpha)$ is generated by a set of relations which Fulton calls the dual Garnir relations.

The dual Garnir relations are constructed as follows. 
For a fixed column $c$ of a tableau $t$ of shape $\lambda$, and for $1\leq k \leq \lambda^{'}_{c+1}$, let $\pi_{c,k}(t)$ be the sum of column tabloids obtained from all possible ways of exchanging the top $k$ elements of the $(c+1)^{st}$ column of $t$ with any subset  of size $k$ of the elements of column $c$, and fixing all other elements of $t$. For example, for 
\[ t = \ytableausetup{centertableaux}
\begin{ytableau}
1 & 4  & 6\\
2 & 5 \\
3
\end{ytableau}  \]
we have 
\[ \pi_{1,1}(t) =  \hspace{.5em} \begin{array}{|c|c|c|}
  \color{red} 4 & \color{red} 1 & 6 \\
  2 & 5 \\
  3
\end{array} \hspace{.5em} +  \hspace{.5em} \begin{array}{|c|c|c|}
  1 & \color{red} 2 & 6\\
  \color{red} 4 & 5 \\
  3
\end{array} \hspace{.5em} +  \hspace{.5em} \begin{array}{|c|c|c|}
  1  & \color{red} 3 & 6\\
  2 & 5 \\
  \color{red} 4
\end{array}\]

Then the \emph{dual Garnir relation} $g_{c,k}(t)$ is
\begin{equation} \label{garnir} g_{c,k}(t) = [t] - \pi_{c,k}(t) .\end{equation}

Note that $t$ can be any tableau, not necessarily with increasing columns. The relation $g_{c,k}(t)$ is called a dual Garnir relation, and varying over $c$ and $k$ gives a straightening algorithm for column tabloids.

\subsection{The Specht module presentation}
The relations $g_{c,k}$ can be used to give several presentations of the Specht module $S^{\lambda}$, the first of which is the following:
\begin{thm}{\normalfont \cite[Ch. 7, Prop. 4]{fulton}} \label{fultonthm}
Let $\tilde{G}^{\lambda}$ be the subspace of $\tilde{M}^{\lambda}$ generated by $g_{c,k}(t)$ where $t$ varies across all $t \in \mathcal{T}_{\lambda}$, $1 \leq c \leq \lambda_1 -1$ and $1 \leq k \leq \lambda^{'}_{c+1}.$ Then
the kernel of $\alpha$ is $\tilde{G}^{\lambda}$. Thus, 
\[ S^{\lambda} \cong \tilde{M}^{\lambda} /\tilde{G}^{\lambda} .  \]
\end{thm}
 A corollary to this theorem is another proof of the classical result that a basis of $S^{\lambda}$ is given by polytabloids of standard Young tableaux of shape $\lambda.$ Our main contribution to this theory will be to give a new presentation of $S^{\lambda}$ that  
reduces the number of generators for $\tilde{G}^{\lambda}$ even further.

\subsubsection{Further simplifications}\label{sec:otherpresentations}

There are two presentations for $S^{\lambda}$ that simplify Theorem \ref{fultonthm}. The first simplification is due to Fulton, who shows in an exercise in \cite[Ch. 7, Ex. 16]{fulton} that the presentation in Theorem \ref{fultonthm} can be simplified further using only the $g_{c,1}$ relations rather than $1 \leq k \leq \lambda^{'}_{c+1}$. Let $\tilde{I}^{\lambda}$ be the subspace generated by $g_{c,1}(t)$ for all $t \in \mathcal{T}_{\lambda}$ and $1 \leq c \leq \lambda_1 -1$. Then this is equivalent to the statement that 
\[  S^{\lambda} \cong \tilde{M}^{\lambda} /\tilde{I}^{\lambda}.  \]

The second simplification is due to Friedmann-Hanlon-Stanley-Wachs in \cite{FHSW} and applies to the case that $\lambda$
is a staircase partition, i.e. $\lambda$ has a conjugate of the form $(n, n-1, n-2, \cdots, n-r$ for some $r$). Their presentation is similar to the one we will give in Theorem \ref{spechtgarnir} in the sense that the number of relations needed to generate $S^{\lambda}$ is analogously reduced. In particular, let $\mathcal{T}^{*}_{\lambda}$ be the set of tableau with increasing columns where each element in $\{ 1, 2, \cdots, |\lambda| \}$ appears exactly once\footnote{Note this is equivalent to enumerating over $[t] \in \tilde{M}^{\lambda}.$}.
\begin{thm}{\normalfont\cite[Cor. 3.2]{FHSW}} \label{thm:FHSWpres}
Let $\lambda$ be a staircase partition and $\tilde{J}^{\lambda}$ be the subspace of $\tilde{M}^{\lambda}$ generated by $g_{c,\lambda^{'}_{c+1}}(t)$ for $t \in \mathcal{T}^{*}_{\lambda}$ and $1 \leq c \leq \lambda_{1}-1$. Then the kernel of $\alpha$ is $\tilde{J}^{\lambda}$. Thus
\[ S^{\lambda} \cong \tilde{M}^{\lambda} /\tilde{J}^{\lambda} .  \]

\end{thm}
The motivation for this presentation comes from the free LAnKe, a generalization of the free Lie algebra. In \cite{FHSW}, the authors initiate the study of the free LAnKe; for more on the connection between the LAnKe and presentations of Specht modules, see \cite[\S 3]{FHSW}. Theorem \ref{thm:FHSWpres} and its connection to the free LAnKe were the initial motivation for this work, and we will employ similar methods to the ones developed in \cite{FHSW}. We will see in Proposition \ref{kernelalpha} that our presentation in Theorem \ref{spechtgarnir} implies the presentation in Theorem \ref{thm:FHSWpres}, and therefore also the CataLAnKe Theorem \cite[Thm. 1.3]{FHSW}.

\section{A simplified presentation} \label{newSp}
In this section we derive a presentation of $S^\lambda$ for $\lambda$ of any shape that requires fewer relations than those needed in Theorem \ref{fultonthm}.

We begin by narrowing our study to partitions $\mu$ of $n+m$ with shape $2^{m}1^{n-m}$, so $\mu$ has a column of size $n$ and a column of size $m$ for $1 \leq m \leq n$, and $\mu'=(n,m)$. We shall generalize these results to partitions of any shape at the end of this section.

Note that by the antisymmetry of column tabloids, $\tilde{M}^{\mu}$, viewed as an $S_{n+m}$-module, can be induced from the Young subgroup $S_{n} \times S_{m} \leq S_{n+m}$ as follows:

\[ \tilde{M}^{\mu}  \cong \ind_{S_{n} \times S_{m}}^{S_{n+m}} \left ( \sgn_{n} \otimes \sgn_{m} \right ) \cong \bigoplus_{i=0}^{m} S^{2^{i}1^{n+m - 2i}} .\]

We will introduce a relation on tableaux that can be thought of as a sum of $\pi_{1,1}$ relations, and which has the advantage of symmetrizing over all positions of elements in the second column. Doing so allows us to restrict our study to tableaux with ordered columns, and to therefore define a new linear transformation from $\tilde{M}^{\mu}$ to $\tilde{M}^{\mu}$.
\begin{definition} \label{defeta}
For $\mu = 2^m1^{n-m}$, let $\eta: \tilde{M}^{\mu} \to \tilde{M}^{\mu}$ be the map
\[\eta[t] = m[t] - \sum [s] \]
where the sum ranges over all possible tableaux $s$ obtained from $t$ by swapping one entry in the second column of $t$ with one entry in the first column. 
\end{definition}
For example, let 
\[ [t] =  \hspace{.25em} \begin{array}{|c|c|}
  1 & 4\\
  2 & 5 \\
  3 
\end{array}. \]
Then
\[ \eta([t]) =  2 \hspace{.5em}  \begin{array}{|c|c|}
  1 & 4\\
  2 & 5 \\
  3 \end{array} \hspace{.25em} - \hspace{.25em}
  \left( \hspace{.25em} \begin{array}{|c|c|}
  \color{red} 4 & \color{red} 1\\
  2 & 5 \\
  3
\end{array}\hspace{.25em} + \hspace{.25em} \begin{array}{|c|c|}
  1 & \color{red} 2\\
  \color{red} 4 & 5 \\
  3
\end{array}\hspace{.25em} + \hspace{.25em} \begin{array}{|c|c|}
  1  & \color{red} 3\\
  2 & 5 \\
  \color{red} 4
\end{array} 
\hspace{.25em}+\hspace{.25em}  \begin{array}{|c|c|}
  \color{red} 5 & 4\\
  2 & \color{red} 1 \\
  3
\end{array}\hspace{.25em} +\hspace{.25em} \begin{array}{|c|c|}
  1 & 4\\
  \color{red} 5 & \color{red} 2 \\
  3
\end{array} \hspace{.25em}+\hspace{.25em} \begin{array}{|c|c|}
  1  & 4\\
  2 & \color{red} 3 \\
  \color{red} 5
\end{array} \hspace{.25em}   \right) . \]

It is clear that $\eta$ is $S_{n+m}$-equivariant. Furthermore, it follows from Theorem \ref{fultonthm} that $\im(\eta) \subseteq \ker(\alpha)$, as $\eta([t])$ is a sum of dual Garnir relations. Using a technique employed in \cite[Sec. 2]{FHSW}, we will now show that the relations generated by $\eta$ are all that is needed to generate $\tilde{G}^{\mu}$.  

\begin{thm} \label{imetakeralpha}
For $\mu = 2^m1^{n-m}$, $\ker(\eta) \cong S^\mu$, and thus $\im(\eta) = \ker(\alpha)$ for $\alpha : \tilde{M}^\mu \rightarrow S^\mu$. 
\end{thm}

Note that because 
\[ \tilde{M}^{\mu} \cong \bigoplus_{i=0}^{m} S^{2^{i}1^{n+m - 2i}}\]
is multiplicity-free, by Schur's Lemma $\eta$ acts as a scalar on each irreducible submodule of $\tilde{M}^{\mu}$. Thus, finding the kernel of $\eta$ is equivalent to finding the irreducible submodules of $\tilde{M}^{\mu}$ on which $\eta$ acts like the 0 scalar.

We proceed by computing the action of $\eta$ on each irreducible submodule of $\tilde{M}^{\mu}$. For each $T \in \binom{[n+m]}{n}$, let $v_{T}\in \tilde{M}^{\mu}$ be the column tabloid with first column $T$ (both columns assumed to be in increasing order). For any $v\in \tilde{M}^{\mu}$, let $\langle v, v_T \rangle$ be the coefficient of $v_T$ in the expansion of $v$ in the basis of all $v_T$. 
\begin{lem}\label{lemma}
For every $S, T \in \binom{[n+m]}{n}$,  
\[ \langle \eta(v_{S}), v_{T} \rangle = \begin{cases}
m & \textrm{if } S = T \\
0 & \textrm{if } |S \cap T| < n-1 \\
(-1)^{x+y} & \textrm{if } |S \cap T| = n -1 \textrm{ with}\\
&  x \in S \backslash T, y \in T \backslash S

\end{cases} \]
\end{lem} 
\begin{proof}
The first two cases follow easily from the definition of $\eta$. For the last case, suppose $x$ is in the $r_{x}^{th}$ row in the first column of $v_{S}$ and $y$ is in the $r_{y}^{th}$ row of the second column of $v_{S}$. Then there are precisely $x-1$ numbers smaller than $x$ altogether, with $r_{x}-1$ of them in the first column. It follows that there are $x - r_{x}$ numbers smaller than $x$ in the second column. Similarly, there are $r_{y} - 1$ numbers smaller than $y$ in the second column and $y - r_{y}$ numbers smaller than $y$ in the first column. 

There are two cases: $x< y$ or $y < x$. 
Suppose $x<y$ and swap the positions of $x$ and $y$. Then in order to obtain an element in the basis of $\tilde{M}^{\lambda}$, we must move $y$ to the $(y-r_{y})^{th}$ row of the first column and $x$ to the $(x - r_{x} +1)^{st}$ row of the second column. This means moving $y$ from the $r_{x}^{th}$ row down to the $(y-r_{y})^{th}$ row, which requires $y - r_{y} - r_{x}$ transpositions. Similarly moving $x$ up from the $r_{y}^{th}$ row to the $(x - r_{x} + 1)^{st}$ row requires $r_{y} - x + r_{x} - 1$ transpositions. Altogether, this amounts to a sign change of 
\[ (-1)^{r_{y} - x + r_{x} - 1 + y - r_{y} - r_{x}} = (-1)^{y-x-1} .\] 
Finally, taking into account that $\eta$ itself contributes a sign change of $(-1)$, we obtain the coefficient $(-1)^{x+y}$ for $\langle \eta(v_{S}), v_{T} \rangle$. The case $y<x$ is similar.
\end{proof}

We next calculate the scalar action of $\eta$ on the irreducible submodules of $\tilde{M}^{\mu}$. 

\begin{thm} \label{eta}
On the irreducible submodule of $\tilde{M}^{\mu}$ isomorphic to $S^{2^{i}1^{(n+m) - 2i}}$, the operator $\eta$ acts like a scalar $\omega_{i}$, where
\[ \omega_{i} := (n+1)(m-i) .\]
\end{thm}
\begin{proof}
For simplicity, take $T = [n]$. 
Then for a given $i$, we take $t$ to be the Young tableau given by
\begin{center}
    \includegraphics[scale = .45]{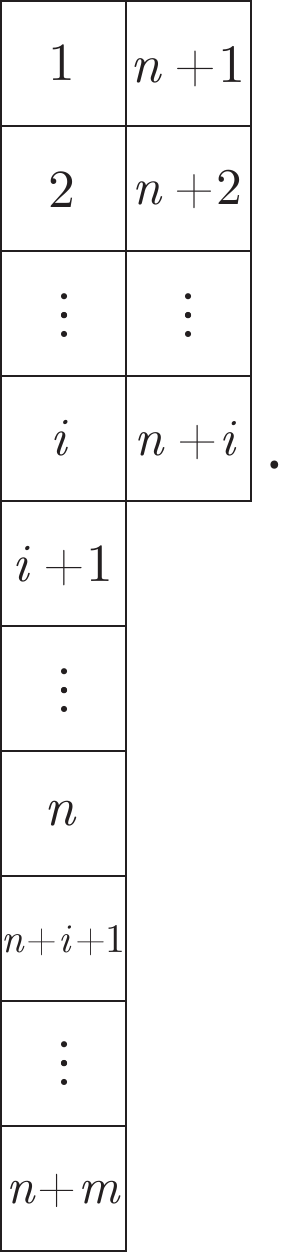}
\end{center}
Recall that $C_{t}$ is the column stabilizer of $t$ and $R_{t}$ is the row stabilizer of $t$. Then we denote by $e_{t}$ the Young symmetrizer of $t$:
\[ e_{t} = \sum_{\alpha \in R_{t}} \sum_{\beta \in C_{t}} \sgn(\beta)\alpha \beta. \]
As in \cite[Thm 2.4]{FHSW}, we adopt a slight abuse of notation by referring to the restriction to $S^{2^{i}1^{n+m-2i}}$ of the space spanned by $\tau e_{t}v_{T}$ for $\tau \in S_{n+m}$ to be $S^{2^{i}1^{n+m-2i}}$ itself.

We define $d_{t}$, $f_{t}$ and $r_{t}$ as in \cite[Thm 2.4]{FHSW}; that is $r_{t} = \sum_{\alpha \in R_{t}} \alpha$, while $d_{t}$ is the signed sum of column permutations stabilizing $\{ 1, 2, \dots , n \}, \{ n+1, \dots, n+i \}$, and $\{ n+i+1, \dots n+ m \}$, and $f_{t}$ is the signed sum of permutations in $C_{t}$ that maintain the vertical order of these sets. Then $e_{t}v_{T} = r_{t}f_{t}d_{t}v_{T}$. 

The antisymmetry of column tabloids ensures that $d_{t}v_{T}$ is a scalar multiple of $v_{T},$ because it simply permutes within columns. Therefore we can conclude that $r_{t}f_{t}v_{T}$ is a scalar multiple of $e_{t}v_{T}$, and in particular that $e_{t}v_{T}$ is nonzero, as the coefficient of $v_{T}$ in $r_{t}f_{t}v_{T}$ is 1.
Consider $\eta(r_{t}f_{t}v_{T}).$ In the subspace restricted to $S^{2^{i}1^{n+m-2i}}$, the fact that $\eta$ acts on $e_{t}v_{T}$ as a scalar implies the same is true of $r_{t}f_{t}v_{T}.$ In fact, because the coefficient of $v_{T}$ in $r_{t}f_{t}v_{T}$ is 1, we can determine precisely what this scalar is by computing $\langle \eta(r_{t}f_{t}v_{T}), v_{T} \rangle$. In particular, we wish to show that

\[ \langle \eta(r_{t}f_{t}v_{T}), v_{T} \rangle = \omega_{i} = (n+1)(m-i) .\]

We have 
\[ r_{t}f_{t}v_{T} = \sum_{S \in \binom{[n+m]}{n}} \langle r_{t}f_{t}v_{T}, v_{S} \rangle v_{S} .\]
Applying the linear operator $\eta$ thus gives 
\[ \eta(r_{t}f_{t}v_{T}) = \sum_{S \in \binom{[n+m]}{n}} \langle r_{t}f_{t}v_{T}, v_{S} \rangle \eta(v_{S}) .\]

Note that when $T = S$, by Lemma \ref{lemma} we have $\langle \eta(v_{T}), v_{T} \rangle = m$. With this, we can compute the coefficient of $v_{T}$ in general by

\begin{align} \label{equation}
\hskip -3em \langle \eta (r_{t} f_{t} v_{T}), v_{T} \rangle =  \sum_{S \in \binom{[n+m]}{n}} \langle r_{t}f_{t}v_{T}, v_{S} \rangle \langle \eta(v_{S}), v_{T} \rangle  = m + \sum_{S \in \binom{[n+m]}{n} \backslash \{ T \} }  \langle r_{t}f_{t}v_{T}, v_{S} \rangle \langle \eta(v_{S}), v_{T} \rangle .
\end{align}

By Lemma \ref{lemma}, for $T \neq S$, $\langle \eta(v_{S}), v_{T} \rangle \neq 0$ only when $S$ and $T$ differ by a single element. In the sum $r_{t}f_{t}v_{T}$, there are two types of possible $v_{S}$ that fulfill this criterion. 
\begin{enumerate}
\item \emph{We can obtain $v_{S}$ from a single row swap.} That is, up to signs, $v_{S}$ is given by $(j, n+j)v_{T}$ for $1 \leq j \leq i$, so $(j, n+j) \in R_{t}$. In this case, in order to write $(j, n+j)v_{T}$ in our basis, we must move $j$ from the $j^{th}$ row to the $1^{st}$ row of the second column and $n+j$ from the $j^{th}$ row to the $n^{th}$ row of the first column. In total, this gives a sign change of $(-1)^{j-1+n-j}= (-1)^{n-1}$. 
By Lemma \ref{lemma}, for such a $v_{S}$, we get $\langle \eta(v_{S}), v_{T}) \rangle = (-1)^{n+j +j} = (-1)^{n}$. Hence overall we get a contribution to Equation \ref{equation} of 
\[   \langle r_{t}f_{t}v_{T}, v_{S} \rangle \langle \eta(v_{S}), v_{T} \rangle  = (-1)^{n-1 + n} = -1. \]
There are $i$ such possible $v_{S}$. Therefore this case contributes $-i$ to Equation \ref{equation}. 

\item \emph{We can obtain $v_{S}$ by a swap coming from a column permutation $\sigma$ in $f_{t}$.} Note that because $f_{t}$ maintains the order of $\{ 1, 2, \dots , n \}, \{ n+1, \dots, n+i \}$, and $\{ n+i+1, \dots n+ m \}$ and we require that $|S \cap T| = n-1$, it must be that $S = \{ 1, 2, \dots, n-1, n+i+1 \}$. Suppose $\sigma$ moves $n$ to the $(n+\ell)^{th}$ row of $t$ for $1 \leq \ell \leq m-i$ and $n+i+1$ to the $k^{th}$ row of $t$ for $1 \leq k \leq n$. To calculate the sign of $\sigma$, write $\sigma = \sigma_{2}\sigma_{1}$, where 
\[ \sigma_{1} =  (n, n+ i + \ell)(n, n+i + \ell -1) \dots (n, n+ i + 1) \]
moves $n$ to the $n+\ell^{th}$ row of $t$ and then
\[ \sigma_{2} = (n+i-l,k)(n+i-1,k+1) \dots (n+i-1, n-2)(n+i-1, n-1) \] 
moves $n+i-1$ to the $k^{th}$ row of $t$.
It follows that $\sgn(\sigma_{1}) = \ell$ and $\sgn(\sigma_{2}) = n-k-1$, so $\sgn(\sigma) = \ell + n-k-1$. 
 
In $\sigma v_{T}$, $n$ is in the $(i + \ell)^{th}$ row of the second column and $n+i+1$ is in the $k^{th}$ row of the first column. In order to put this in our basis, we must move $n$ to the first row in the second column and $n+i+1$ to the $n^{th}$ row of the first column, which requires $i + \ell - 1$ transpositions and $n-k-1$ transpositions, respectively. 
Combining these gives a sign change of 
\[ (-1)^{(\ell + n-k-1) +(i + \ell -1 +n-k-1)} = (-1)^{i-1} .\]
For such a $v_{S}$, the coefficient of $\langle \eta(v_{S}), v_{T} \rangle$ is  $(-1)^{n+i +1 +n} = (-1)^{i+1}$.  

Thus for such a $v_{S}$ and $\sigma$, we get a total coefficient of $(-1)^{i+1 + i -1} = 1.$
There are $m-i$ choices for $\ell$ and $n$ choices for $k$, giving $n(m-i)$ possible $\sigma$. Hence we get a contribution to Equation \ref{equation} of $n(m-i)$.  

\end{enumerate}

Thus combining the $T=S$ case with the two cases above, we have 
\[ \omega_i=\langle \eta (r_{t} f_{t} v_{T}), v_{T} \rangle = m + (- i) + n(m-i) = (n+1)(m-i) . \] \vskip -.7cm \end{proof} 

\begin{proof}[Proof of Theorem \ref{imetakeralpha}] By Theorem \ref{eta}, $\omega_i$ is 0 only when $m=i$, so $\ker(\eta) \cong S^{\mu}$. Thus $\im(\eta) \cong \tilde{M}^{\mu} / \ker(\eta) = \ker(\alpha)$, and the theorem is proved.
\end{proof}

Theorem \ref{imetakeralpha} allows us to generate $\tilde{G^{\mu}}$ for any $\mu$ with two columns using only the single $\eta$ relation. 

We now consider any partition $\lambda = (\lambda_{1}, \dots , \lambda_{k})$ with conjugate $\lambda^{'} = (\lambda^{'}_{1}, \dots , \lambda^{'}_{j})$. 
For $t \in \mathcal{T}_{\lambda}$, let $h_{c}([t])$ be the image of $\eta$ on the $c$ and $(c+1)^{st}$ columns of $t$ that leaves the other columns of $[t]$ fixed.  
\begin{thm}[New Presentation of Specht Modules.]\label{spechtgarnir}
For any partition $\lambda$ of $m$, let $\tilde{H}^{\lambda}$ be the space generated by $h_{c}([t])$ for every $[t] \in \tilMlam$ and $1 \leq c \leq \lambda_{1}-1$. Then the kernel of $\alpha$ is $\tilde{H}^{\lambda}$. Thus
\[ S^{\lambda} \cong \tilde{M}^{\lambda} / \tilde{H}^{\lambda} .\]
\end{thm}
The proof of Theorem \ref{spechtgarnir} follows from Theorem \ref{imetakeralpha} and the definition of $h_{c}([t])$. \\

Theorem \ref{spechtgarnir} dramatically reduces the number of generators needed to find $\tilde{G}^{\lambda}$. The original construction of Theorem \ref{fultonthm} required enumerating over every $1 \leq k \leq \lambda _{c+1}'$ for every pair of columns $c$ and $c+1$ of every $t \in \Tlam$. Even Fulton's simplification using only $g_{c,1}$ relations requires enumerating over $t \in \Tlam$ for every pair of columns $c$ and $c+1$. By contrast, our construction uses a single relation for every pair of adjacent columns, and $[t]$ varies in $\tilMlam$, a significantly smaller space than $\Tlam$. 

\subsection{Implications for other presentations of Specht modules}
We now address the relationship between the presentation in Theorem \ref{spechtgarnir} and the presentations discussed in $\S \ref{sec:otherpresentations}$.

\subsubsection{Fulton's simplified presentation}
Note that because $\eta([t])$ is a sum of relations of the form $g_{c,1}(t)$, we have the containment
\[ \tilde{H}^{\lambda} \subseteq \tilde{I}^{\lambda} \subseteq \tilde{G}^{\lambda}.\]
Theorem \ref{spechtgarnir} shows that $\tilde{H}^{\lambda} = \tilde{G}^{\lambda}$, which forces equality between all three subspaces. It follows that Theorem \ref{spechtgarnir} actually implies Fulton's simplified presentation of $S^{\lambda}$ as the quotient $\tilde{M}^{\lambda} / \tilde{I}^{\lambda}$.

\subsubsection{Friedmann-Hanlon-Stanley-Wachs presentation}
It turns out that Theorem \ref{spechtgarnir} also implies Theorem \ref{thm:FHSWpres}, although it will take a bit more work to show. Fix a partition $\mu = 2^{n-1}1$. 

We will first write the relation $g_{1,n-1}$ as a map from the space $\tilde{M}^{\mu}$ to itself.\footnote{The Garnir relations are defined to take as input a \emph{tableau} and output a sum of column tabloids.} Recall that $v_{S}\in \tilde{M}^{\mu}$ is the column tabloid with first column $S$ for $S\in \binom{[2n-1]}{n}$, where both columns are assumed to be in increasing order. 

Again, for simplicity let $T = [n]$ and $R_{i} = \{ n+1, \dots 2n-1, i \} $.
Following the notation in \cite{FHSW}, define 
\[ \varphi(v_{T}) = v_{T} - \sum_{i=1}^{n} (-1)^{i-1} v_{R_{i}}. \]
Note that for $t \in \mathcal{T}_{\mu}$ with first column $T$ (in increasing order), $\varphi(v_{T})$ is precisely $g_{1,n-1}(t)$.

\begin{prop}\label{kernelalpha}
For $\mu = 2^{n-1}1$, we have $\im(\eta) \subseteq \im(\varphi) \subseteq \tilde{G}^{\lambda}.$

\end{prop}

\begin{proof}
With $T = [n]$ and 
for $i\in [n]$ and $j\in [n-1]$, let 
 \begin{align*}
 R_{i} = \{ n+1, \dots 2n-1, i \},\\
  S_{i,j} = \{ 1, \dots \hat{i}, \dots, n, n+j \}. 
 \end{align*}

Then 
\begin{equation*}  
\varphi (v_{T}) =v_{T} - \sum_{i=1}^{n} (-1)^{i-1} v_{R_{i}}; \hskip .5cm 
\varphi(v_{R_{i}}) = v_{R_{i}} - \left (  \sum_{j = 1}^{n-1} (-1)^{j+n-1}  v_{S_{i,j}} + (-1)^{i-1}v_{T} \right ) . 
\end{equation*}
It can then be shown that
\begin{equation} \label{eqneta} 
\eta(v_{T}) = -\left ( \varphi(v_{T}) + \sum_{i = 1}^{n} (-1)^{i-1} \varphi(v_{R_{i}}) \right )   ,\end{equation}
from which the claim follows. 
\end{proof}
Note that Theorem \ref{imetakeralpha} implies that $\im(\eta) = \tilde{G}^{\lambda}$. Therefore 
Proposition \ref{kernelalpha} in fact shows the equality:
\[ \im(\eta) = \im(\varphi) = \tilde{G}^{\lambda},\]
which implies the presentation in Theorem \ref{thm:FHSWpres}.

\bibliographystyle{plain}
\bibliography{bibliography}
\Addresses
\appendix

\end{document}